\documentclass[reqno]{amsart}%

\usepackage{amsmath}
\usepackage{amsfonts}
\usepackage{CJK}
\usepackage{eufrak}
\usepackage{amssymb}
\usepackage{graphicx}%

\usepackage{hyperref}

\usepackage{graphicx}%

\usepackage[centertags]{amsmath}
\usepackage{amsthm}
\usepackage{supertabular}
\usepackage{graphics,graphicx}
\usepackage{newlfont}
\usepackage{amssymb,latexsym,amsthm,xypic}
\usepackage{color}
\usepackage{fancyhdr}%%ҳü��ʽ��

\newcommand{\mc}{\mathbb{C}}
\newcommand{\mr}{\mathbb{R}}

\newcommand{\bi}{\begin{itemize}}
\newcommand{\ei}{\end{itemize}}

\newtheorem{thm}{Theorem}[section]

\newtheorem{lem}{Lemma}[section]
\newtheorem{prop}[thm]{Proposition}

\theoremstyle{definition}

\theoremstyle{Conjecture}

\theoremstyle{remark}

\theoremstyle{Example}

\newcommand{\be}{\begin{equation}}
\newcommand{\ee}{\end{equation}}
\newcommand{\bs}{\begin{split}}
\newcommand{\es}{\end{split}}
\newcommand{\bea}{\begin{eqnarray}}
\newcommand{\eea}{\end{eqnarray}}
\newcommand{\ben}{\begin{eqnarray*}}
\newcommand{\een}{\end{eqnarray*}}
\newcommand{\bet}{\begin{equation}

\begin{split}}
\newcommand{\eet}{\end{split}
\end{equation}}

\begin{document}
\title[]{Flat bundles over some compact complex manifolds}
\date{}
\author[F. Deng, J. E. Forn{\ae}ss]{Fusheng Deng, John Erik Forn{\ae}ss}
\address{F. Deng: School of Mathematical Sciences, University of Chinese Academy of Sciences, Beijing 100049, China}
\email{fshdeng@ucas.ac.cn}
\address{J. E. Forn{\ae}ss: Department for Mathematical Sciences, Norwegian University of Science
and Technology, Trondheim, Norway}
\email{john.fornass@math.ntnu.no}

\begin{abstract}
We construct examples of flat fiber bundles over the Hopf surface
such that the total spaces have no pseudoconvex neighborhood basis,
admit a complete K\"ahler metric, or are hyperconvex but have no nonconstant holomorphic functions.
For any compact Riemannian surface of positive genus,
we construct a flat $\mathbb P^1$ bundle over it and a Stein domain with real analytic bundary in it
whose closure does not have pseudoconvex neighborhood basis.
For a compact complex manifold with positive first Betti number,
we construct a flat disc bundle over it such that the total space is hyperconvex but admits no nonconstant holomorphic functions.
\end{abstract}
\maketitle

\section{Introduction}
The aim of the present note is to construct flat bundles over some compact complex manifolds
such that the total spaces can be used as examples for some problems in several complex variables.

It is known there exists a bounded pseudoconvex domain $D$ in $\mc^n$ ($n>1$) with smooth boundary
such that its closure has no pseudoconvex (or Stein) neighborhood basis \cite{Diederich-Fornaess77}.
But if a bounded pseudoconvex domain in $\mc^n$ has real analytic boundary,
its closure has a pseudoconvex neighborhood basis \cite{Diederich-Fornaess78}.
So it is natural to ask whether the closure of a bounded pseudoconvex or Stein domain with real analytic boundary in a complex manifold
has apseudoconvex neighborhood basis.
In this note, we will construct a domain $\Omega$ with real analytic Levi-flat boundary in the total space $B$ of a flat $\mathbb P^1$ bundle over a Hopf surface
such that its closure does not have a pseudoconvex neighborhood basis, with $\Omega$ realized as a flat disc bundle over the Hopf surface.
In this example, $B$ is not projective and $\Omega$ is not Stein.
For any compact Riemannian surface of positive genus,
we also construct a flat $\mathbb P^1$ bundle over it such that the total space is a projective manifold,
and a Stein domain with real analytic boundary in the total space whose closure does not have a pseudoconvex neighborhood basis.

The first domain $\Omega$ mentioned in the previous paragraph is a flat disc bundle over the Hopf surface.
Note that the Hopf surface is not K\"ahler.
It is natural to ask whether $\Omega$ admits a K\"ahler metric.
Using the Oka-Grauert principle, we will show that $\Omega$ is biholomorphic to
the product $A\times\mc^2\backslash 0$, where $A$ is an annulus,
and hence admits a complete K\"ahler metric.
We also construct a flat $\mc^*$-bundle over the Hopf surface such that the total
space has a complete K\"ahler metric.

Finally, for any compact complex manifold with positive first Betti number,
we construct a flat disc bundle over it such that the total space has a bounded exhaustive
real analytic plurisubharmonic function but does not have any nonconstant holomorphic function.
The method can be generalized to construct for any $n>1$ an $n$-dimensional complex manifold $X$
such that $X$ has a real analytic bounded plurisuhbarmonic exhaustion whose complex Hessian has $n-1$ positive eigenvalues everywhere
but $X$ has no nonconstant holomorphic functions.
Examples of hyperconvex domains without bounded holomorphic functions was considered in \cite{Adachi17}.

\vspace{.1in} {\em Acknowledgements} The first author is partially supported by NSFC grants.
The second author partially supported by an NFR grant.

\section{The general framework}%\label{sec:famework}
In the section, we set up a general framework for constructing flat fiber bundles,
and prove a result about the existence of complete K\"ahler metrics that will be used a few times later.

\subsection{Construction of flat bundles}\label{subsec:framework}
Let $\tilde X$ be a complex manifold.
Let $\Gamma\subset Aut(\tilde X)$ be a discrete subgroup in the automorphism group $Aut(\tilde X)$ of $\tilde X$
that acts freely on $\tilde X$. Then $X:=\tilde X/\Gamma$ is a complex manifold.
Let $F$ be another complex manifold.
Let $\sigma:\Gamma\rightarrow Aut(F)$ be a group morphism.
We can construct a flat fiber bundle $p:B\rightarrow X$ with fiber $F$ as follows, where flatness means that the transition functions are locally constant.
First let $\tilde B=\tilde X\times F$ be the trivial bundle over $\tilde X$ and $\tilde p:\tilde B\rightarrow \tilde X$ be the natural projection.
Then we can define an action of $\Gamma$ on $\tilde B$ as follows:
$$\gamma\cdot(z,v)=(\gamma z, \sigma(\gamma)v),\  (z,v)\in \tilde X\times F.$$
It is obvious that the action of $\Gamma$ on $\tilde B$ is freely and properly discontinuously.
The quotient $B:=\tilde B/\Gamma$ is a complex manifold and gives a flat fiber bundle over $X$ with fibers biholomorphic to $F$.
We have the following commutative diagram:
$$
\begin{xy}
(0,20)*+{\tilde B}="v1";(20,20)*+{B}="v2";
(0,0)*+{\tilde X}="v3";(20,0)*+{X}="v4";%
{\ar@{->}^{\tilde\pi} "v1"; "v2"};{\ar@{->}_{\tilde p} "v1";"v3"};
{\ar@{->}^{p}"v2"; "v4"};{\ar@{->}^{\pi} "v3"; "v4"};;
\end{xy},
$$
where $\tilde \pi$ is the quotient map, and $p$ is defined as $p(\tilde\pi(z,v))=\pi(z)$.

Let $v\in F$ and let $s_v:\tilde X\rightarrow \tilde B$ be the constant section given by $s_v(z)=(z,v)$ for all $z\in \tilde X$.
We consider the map $\tilde{\pi}\circ s_v:\tilde X\rightarrow B$ and let $X_v:=\tilde{\pi}\circ s_v(\tilde X)\subset B$.
Note also that there is an action of $\Gamma$ on $F$ induced by $\sigma:\Gamma\rightarrow Aut(F)$.
We have the following:

\begin{lem}\label{lem:loc constant section}
With the above notations, we have
\begin{itemize}
\item[(1)] If the action of $\Gamma$ on $F$ is free,
then the map $\tilde\pi\circ s_v:\tilde X\rightarrow B$ is injective;
\item[(2)] If the action of $\Gamma$ on $F$ is properly discontinuous,
then the map $\tilde\pi\circ s_v:\tilde X\rightarrow B$ is proper;
\item[(3)] If the action of $\Gamma$ on $F$ is free and properly discontinuous,
then the map $\tilde\pi\circ s_v:\tilde X\rightarrow B$ is a proper holomorphic embedding;
\item[(4)] If $v\in F$ is fixed by $\Gamma$, then $X_v$ is biholomorphic to $X$;
\item[(5)] For any $v,w\in F$, if $X_v\cap X_w\neq \emptyset$, we have $X_v=X_w$; and $X_v=X_w$ if and only if $w=\sigma(\gamma)v$ for some $\gamma\in \Gamma$.
\end{itemize}
\end{lem}
\begin{proof}
For simplicity, we denote $\tilde\pi(z,v)$ by $[z,v]$ for $(z,v)\in\tilde B$.
\begin{itemize}
\item[(1)] if $[z,v]=[z',v]$, there is some $\gamma\in \Gamma$ such that $z'=\gamma z$ and $v=\sigma(\gamma)v$.
We then have $\gamma=Id$ and hence $z=z'$ since the action of $\Gamma$ on $F$ is free.
So $\tilde\pi\circ s_v$ is injective.
\item[(2)] Recall that the action of $\Gamma$ on $F$ is called properly discontinuous if for any compact sets $K, L\subset F$,
the set $\{\gamma\in \Gamma; \gamma K\cap L\neq\emptyset\}$ is always finite.
we need to prove that if $z_n\rightarrow \infty$ in $\tilde X$,
then $[z_n, v]\rightarrow \infty$ in $B$.
If it is not the case, then we can find a sequence $\{z_n\}$ in $\tilde X$
such that $z_n\rightarrow\infty$ but $\{[z_n,v]\}$ lies in some compact set $K$ of $B.$
Let $K_1\subset\tilde X, K_2\subset F$ be compact sets such that $K\subset\tilde\pi(K_1\times K_2)$.
Then for each $z_n$, we can find $\gamma_n\in\Gamma$ such that $\gamma_n(z_n)\in K_1, \sigma(\gamma_n)(v)\in K_2$.
From the first inclusion, we see that $\{\gamma_n;n\geq 1\}$ is infinite;
but from the second inclusion, we see that $\{\gamma_n;n\geq 1\}$ is a finite set since the action of $\Gamma$ on $F$ is properly discontinuously.
Contradiction.
\item[(3)] this is a direct consequence of (1) and (2).
\item[(4)] this is obvious.
\item[(5)] assume $X_v\cap X_w\neq\emptyset$, then we have $[z_1,v]=[z_2,w]$ for some $z_1,z_2\in\tilde X$.
this implies $(z_2,w)=(\gamma z_1, \sigma(\gamma)v)$ for some $\gamma\in\Gamma$.
So for any $z\in\tilde X$, we have $[z,v]=[\gamma z, \sigma(\gamma)v]=[\gamma z, w]\in X_w$.
Hence we have $X_v\subset X_w$. We can prove $X_w\subset X_v$ in the same way.
The second statement in (5) is obvious.
\end{itemize}
\end{proof}

From (5) in the above lemma, we see that $B$ is foliated by $X_v$ for $v\in F$,
and the parameter space of the leaves is naturally identified to the quotient space $F/\Gamma$.

The above construction can be extended to some continuous group actions,
namely, replacing $\Gamma$ by some continuous subgroup of $Aut(\tilde X)$.
But we will not carry out the details of this case in this note.

\subsection{Existence of K\"ahler metrics}\label{subsec:existence Kahler metric}
We now prove a result about the existence of K\"ahler metrics that will be used several times later.

\begin{prop}\label{prop:submersion case}
Let $p:X\rightarrow T$ be a holomorphic submersion, where $X$ and $T$ are compact complex manifolds.
Assume that there is a holomorphic line bundle $L$ over $X$ such that
$L|_{X_t=p^{-1}(t)}$ is ample for all $t\in T$.
Then $X$ is K\"ahler if $T$ is K\"ahler ,
and $X$ is projective if $T$ is projective.
\end{prop}

We first proof a Lemma.

\begin{lem}\label{lem:positive metric fiber}
Let $X, T, L$ as in Proposition \ref{prop:submersion case}.
Then there is a Hermitian metric on $L$ such that $Ric(L,h)|_{X_t}$
is strictly positive for all $t\in T$,
where $Ric(L,h)$ is the Ricci curvature form of $L$ with respect
to $h$.
\end{lem}
\begin{proof}
Replacing $L$ by some powers of it,
we may assume for simplicity that $L$ is very ample on each fiber $X_t$.

Instead of considering $L$ itself, we will consider the dual bundle $L^*$ of $L$.
Fix an arbitrary point $t_0\in T$, there is a neighborhood $U$ of $t_0$ in $T$
such that there is a fiber preserving diffeomorphism
$$\phi:U\times X_{t_0}\rightarrow p^{-1}(U)$$
such that $\phi$ is the identity on $X_{t_0}$.

Let $\tilde L=\phi^{-1}L|_{p^{-1}(U)}$, then $\tilde L$ is a smooth complex line bundle
over $U\times X_{t_0}$ and $\tilde L|_{X_{t_0}}=L|_{X_{t_0}}$.
Since $L|_{X_{t_0}}$ is very ample, there is a holomorphic map $g:X_{t_0}\rightarrow \mathbb P^\infty$
such that $L|_{X_{t_0}}=g^{-1}O(1)$, where $O(1)$ is the dual of the tautological line bundle $O(-1)$ over $\mathbb P^\infty$.
So $\tilde L^*|_{X_{t_0}}=g^{-1}O(-1)$.

We know assume that $U$ is diffeomorphic to a ball.
It is known from topology that $\tilde L^*$ is isomorphic as topological complex line bundles
to $G^{-1}O(-1)$ for some smooth map $G:U\times X_{t_0}\rightarrow \mathbb P^\infty$.
The restriction $G|_{X_{t_0}}$ and $g$ are homotopic.
So $G$ is homotopic to the map $g\circ p_1:U\times X_{t_0}\rightarrow \mathbb P^\infty$,
where $p_1:U\times X_{t_0}\rightarrow X_{t_0}$ is the natural projection.
It follows that $\tilde L^*$ is isomorphic to $p_1^{-1}\tilde L^*|_{X_{t_0}}=p_1^{-1}L^*|_{X_{t_0}}$ as smooth complex line bundles.

There is a canonical Hermitian metric $h_0$ on $O(-1)$ whose curvature form is strictly negative.
This induces a smooth Hermitian metric say $h^*_U$ on $L^*|_{p^{-1}(U)}$.
The curvature form of $(L^*|_{p^{-1}(U)}, h^*_U)$ is strictly negative on $X_{t_0}$.
By continuity and by contracting $U$ is necessary, we see that
the curvature form of $(L^*|_{p^{-1}(U)}, h^*_U)$ is strictly negative on $X_t$ for all $t\in U$.

So there is a finite open cover $\{U_\alpha\}$ of $T$ and Hermitian metrics $h^*_\alpha$'s
on $L^*|_{U_\alpha}$ for each $\alpha$ such that $Ric(L^*|_{U_\alpha}, h^*_\alpha)$ is strictly
negative on $X_t$ for all $t\in U_\alpha$.
Let $\{\rho_\alpha\}$ be a partition of unity of $T$ with respect to the open cover $\{U_\alpha\}$,
then $h^*=\sum_{\alpha}h^*_\alpha$ is a Hermitian metric on $L^*$ such that
$Ric(L^*,h^*)$ is strictly negative along $X_t$ for all $t\in T$.
Let $h$ be the metric on $L$ dual to $h^*$,
then $Ric(L,h)$ is strictly positive along $X_t$ for all $t\in T$.
\end{proof}

We now give the proof of Proposition \ref{prop:submersion case}.

\begin{proof}
We give the proof that $X$ is projective if $T$ is projective.
The proof of the first statement is similar.

By Lemma \ref{lem:positive metric fiber}, there is a Hermitian metric $h$ on $L$
such that the curvature form $Ric(L,h)$ is positive on $X_t=X_{p^{-1}(t)}$ for all $t\in T$.

Let $L_0$ be a positive line bundle over $T$ with a Hermitian metric $h_0$ such that $Ric(L_0,h_0)$
is positive. Then $Ric(L,h)+Np^*Ric(L_0,h_0)$ is positive for $N>>1$.
Let $\tilde L_0=p^{-1}L_0$ be the pull back of $L_0$, which is a line bundle over $X$.
Then $h+Np^*h_0$ is a Hermitian metric on $L+Np^{-1}L_0$ whose curvature form is $Ric(L,h)+Np^*Ric(L_0,h_0)$.
It follows that $L+Np^{-1}L_0$ is a positive line bundle over $X$.
By Kodaria's embedding theorem, $X$ is a projective manifold.
\end{proof}

Let $X$ be a compact complex manifold.
Let $p:E\rightarrow X$ be a holomorphic vector bundle of rank $r$ over $X$.
Let $\{U_\alpha\}_{\alpha\in\Lambda}$ be an open cover of $X$ such that $E|_{U_\alpha}$ is trivial for all $\alpha$.
Let $g_{\alpha\beta}:U_\alpha\cap U_\beta\rightarrow GL(r,\mathbb C)\subset Aut(\mc^r)$ be the transition functions of $E$.
Then $E$ can be constructed by gluing all $U_\alpha\times\mc^r$ via the transition functions $g_{\alpha\beta}$.
We now view $\mc^r$ as a subset of the projective space $\mathbb P^r$.
Then we have a natural inclusion from $GL(r,\mc)$ to $Aut(\mathbb P^r)=PGL(r+1,\mc)$.
So the transition function $g_{\alpha\beta}$ of $E$ induce a fiber bundle $E_\infty$ over $X$ with fibers biholomorphic to $\mathbb P^r$.
We have a natural inclusion $E\subset E_\infty$.

Let $D=E_\infty\backslash E$. Then $D$ is a divisor of $E_\infty$.
Let $L$ be the line bundle over $E_\infty$ associated to the divisor $D$.
Then the restriction of $L$ on each fiber of $E_\infty\rightarrow X$ is positive.
By Proposition \ref{prop:submersion case}, we have the following

\begin{lem}\label{lem:general existece Kahler metric}
If $X$ is a K\"ahler manifold, then $E_\infty$ is a K\"ahler manifold;
if $X$ is a projective manifold, then $E_\infty$ is a projective manifold.
\end{lem}

%
%\begin{rem}\label{lem:positive metric fiber}
%In the proof of Proposition \ref{lem:positive metric fiber},
%the reason for considering  $L^*$ instead of $L$ itself is that the sum of two Hermitian metrics
%with positive curvature  on a line bundle do not have positive curvature in general,
%while the sum of two Hermitian metrics with negative curvature still have negative curvature.
%This fact follows from a simple result in several complex variables.
%Assume that $\phi_1$ and $\phi_2$ are plurisubharmonic functions.
%If $\phi$ is given by
%$$e^\phi=e^{\phi_1}+e^{\phi_2},$$
%then $\phi$ is a plurisubharmonic function;
%but if $\phi$ is given by
% $$e^{-\phi}=e^{-\phi_1}+e^{-\phi_2},$$
%$\phi$ is not plurisubharmonic in general.
%The positivity of the averaging of positively curved Hermitian metrics on a line bundle
%is the central study in the theory of minimum principle for plurisubharmonic functions,
%which is originated by Kiselman in \cite{Kiselman78} and was further studied by other authors
%(see e.g. \cite{Chafi83}\cite{Kiselman78}\cite{Zhou98}\cite{Zhou02}\cite{Berndtsson98}\cite{Berndtsson06}\cite{Deng-Zhou-Zhang14}\cite{Deng-Zhou-Zhang17}).
%\end{rem}

Another result that will be used repeatedly is the following

\begin{lem}\label{lem:Kahler metric complement}
Let $X$ be a compact K\"ahler manifold or a Stein manifold and $A\subset X$ be a closed analytic subset.
Then $X\backslash A$ admits a complete K\"ahler metric.
\end{lem}
For the proof of Lemma \ref{lem:Kahler metric complement}, see \cite{Demailly}.

\section{Flat bundles over the Hopf surface}\label{sec:flat bundle hopf}
\subsection{Disc bundle over the Hopf surface I}\label{subsec:disc bundle hopf:p.s.c nbd basis}
We denote $\mc^2\setminus \{0\}$ by $\mc^2_*$.
Let $\gamma_n\in Aut(\mc^2_*)$ be the map given by $z\rightarrow 2^n z$,
and let $\Gamma=\{\gamma_n; n\in\mathbb Z\}$.
Then $\Gamma$ acts freely and discontinuously on $\mc^2_*$.
The quotient space $H=\mc^2_*/\Gamma$ is called the Hopf surface,
which is a compact complex manifold of dimension 2.
Since $\pi_1(H)\approx \Gamma\approx \mathbb Z$ is commutative,
the first homology group of $H$ is $H_1(H,\mathbb Z)=\pi_1(H)/[\pi_1(H),\pi_1(H)]=\pi_1(H)\approx \mathbb Z$.
So by Hodge theory for compact K\"{a}hler manifolds,
we know that $H$ does not admit any K\"{a}hler metric.

Let $\phi\in Aut(\Delta)$ be an automorphism of the unit disc $\Delta$ given by $\phi(v)=\frac{v+1/2}{1+v/2}$,
and let $\Gamma'=\{\phi^n; n\in\mathbb Z\}$.
Then  $\Gamma'$ also acts freely and properly discontinuously on $\Delta$.
The map $\sigma:\Gamma\rightarrow\Gamma'$ given by $\gamma_n\mapsto \phi^n$ is obviously a group isomorphism.
Since elements in $\Gamma'$ are fractional transformations on the projective line $\mathbb P^1$,
we can also view $\Gamma'$ as a subgroup of $Aut(\mathbb P^1)$.

Now we apply the construction in \S \ref{subsec:framework} to our special case
by setting $\tilde X=\mc^2_*$, $X=H$, and $F=\mathbb P^1$.
Let $\tilde B=\mc^2_*\times \mathbb P^1$ be the trivial $\mathbb P^1$ bundle over $\mc^2_*$.
Then the action of $\Gamma$ on $\mc^2_*$ lifts to an action on $\tilde B$ as follows:
$$\gamma_n\cdot(z,v)=(2^nz,\phi^n(v)),\ (z,v)\in\mc^2_*\times\mathbb P^1.$$
Let $B=\tilde B/\Gamma$ and let $\tilde\pi:\tilde B\rightarrow B$ be the quotient map.
For simplicity, we denote $\tilde\pi(z,v)$ by $[z,v]$ for $(z,v)\in\tilde B$.
Then the map $p:B\rightarrow H$ given by $[z,v]\mapsto \pi(z)$ realize $B$
as a flat $\mathbb P^1$ bundle over $H$,
where $\pi:\mc^2_*\rightarrow H$ is the quotient map.
In summery, we have the following commutative diagram:

$$
\begin{xy}
(0,20)*+{\tilde B}="v1";(20,20)*+{B}="v2";
(0,0)*+{\mc^2_*}="v3";(20,0)*+{H}="v4";%
{\ar@{->}^{\tilde\pi} "v1"; "v2"};{\ar@{->}_{\tilde p} "v1";"v3"};
{\ar@{->}^{p}"v2"; "v4"};{\ar@{->}^{\pi} "v3"; "v4"};;
\end{xy}.
$$

We decompose $\mathbb P^1$ into three $\Gamma'$-invariant pieces $\Delta, \Delta', S^1$,
where $S^1$ is the unit circle, and $\Delta'$ is the complement of the closed unit disc.
It is clear that $\Gamma'$ acts on $\Delta$ and $\Delta'$ freely and properly discontinuously.
The action of $\Gamma'$ on $S^1\backslash\{\pm 1\}$ is also freely and properly discontinuously.
Note that $\pm 1$ are attracting fixed points of $\Gamma'$.
More precisely, we have $\phi(\pm 1)=\pm 1$, and for any $v\in\mathbb P^1\backslash\{\pm 1\}$,
$\lim_{n\rightarrow \pm \infty}\phi^n(v)=\pm 1$.

Let $\tilde\Omega=\mc^2_*\times\Delta, \tilde\Omega'=\mc^2_*\times\Delta', \tilde M=\mc^2_*\times S^1$.
Then $\tilde\Omega, \tilde\Omega'$ are $\Gamma$-invariant open subsets of $\tilde B$,
and $\tilde M$ is a $\Gamma$-invariant Levi-flat hypersurface of $\tilde B$.
Let $\Omega=\tilde\Omega/\Gamma, \Omega'=\tilde\Omega'/\Gamma, M=\tilde M/\Gamma$,
which are subsets of $B=\tilde B/\Gamma$.
Both $\Omega$ and $\Omega'$ are flat disc bundle over $H$, and $M$
is a flat circle bundle over $H$.

For $v\in\mathbb P^1$, as in the previous section,
let $s_v:\mc^2_*\rightarrow \tilde B$ be the constant section given by $s_v(z)=(z,v)$,
and denote by $X_v\subset B$ the image of $\tilde\pi\circ\tilde s_v$.
It is clear that $X_v\subset \Omega, \Omega', \Omega_0$ if $v\in \Delta, \Delta', S^1$ respectively.
By Lemma \ref{lem:loc constant section}, for $v\in \Delta,\ (\Delta')$,
the map $\tilde\pi\circ s_v:\mc^2_*\rightarrow\Omega\ (\tilde\pi\circ s_v:\mc^2_*\rightarrow\Omega')$
is a proper holomorphic embedding.
So we have the following

\begin{thm}[\cite{Diederich-Fornaess82}]
$\Omega$ is a pseudoconvex domain in $B$.
But there does not exist an increasing sequence of relatively compact pseudoconvex domains $\Omega_n$
in $\Omega$ such that $\cup\Omega_n=\Omega$.
\end{thm}
\begin{proof}
The boundary of $\Omega$ in $B$ is $M$, which is a real analytic compact Levi-Flat hypersurface in $B$.
So $\Omega$ is a pseudoconvex domain in $B$.
If there is an increasing sequence of relatively compact pseudoconvex domains $\Omega_n$ such that $\cup\Omega_n=\Omega,$
then $\Omega_n\cap X_0$ is an increasing sequence of relatively compact pseudoconvex domains in $X_0$ such that $\cup(\Omega_n\cap X_0)=\Omega$.
Note that since $X_0$ is biholomorphic to $\mc^2_*$,
this is impossible by Hartogs extension theorem.
\end{proof}

For $v=\pm 1$, it is clear that $X_v\subset M$ are biholomorphic to the Hopf surface $H$.
For $v\in\mathbb P^1\backslash\{\pm 1\}$, by the property of the action of $\Gamma'$ on $\mathbb P^1$,
we see that $\tilde\pi\circ s_v(z)\rightarrow X_1$ as $z\rightarrow\infty$,
and $\tilde\pi\circ s_v(z)\rightarrow X_{-1}$ as $z\rightarrow 0$.

\begin{thm}
The closure $\overline\Omega=\Omega\cup M$ of $\Omega$ in $B$ has no pseudoconvex neighborhood basis.
More precisely, there does not exist a decreasing sequence of pseudoconvex domains $\Omega_n$ in $B$
such that $\cap\Omega_n=\overline\Omega$.
\end{thm}
\begin{proof}
We argue by contradiction.
Suppose that $\Omega_n$ is a decreasing sequence of pseudoconvex domains $\Omega_n$ in $B$
such that $\cap\Omega_n=\overline\Omega$. By Lemma \ref{lem:loc constant section},
$\tilde\pi\circ s_\infty:\mc^2_*\rightarrow\Omega'$ is a proper holomorphic embedding.
We identify $X_\infty$ with $\mc^2_*$ by identifying $[z,\infty]$ with $z$.
Since $\tilde\pi\circ s_\infty(z)\rightarrow X_1\subset M$ as $z\rightarrow\infty$,
$\Omega_n\cap X_\infty$ is a pseudoconvex domain in $\mc^2_*$
that contains a neighborhood of $\infty\in\mc^2$.
However, by Hartogs extension theorem, a pseudoconvex domain in $\mc^2$ that contains a neighborhood of $\infty$
has to be equal to the whole $\mc^2$. Contradiction.
\end{proof}

Let $\Delta_r=\{v\in\mc; |v|<r\}$.
One may wonder whether $\Omega_r:=\tilde\pi(\mc^2_*\times\Delta_r) (r>1)$ is a pseudoconvex neighborhood basis of $\Omega$.
However, from Lemma \ref{lem:loc constant section}, it is easy to see that
$\Omega_r=B$ for any $r>1$.

\subsection{Disc bundle over the Hopf surface II}\label{subsec:disc bundle hopf:Kahler metric}
We will follow the notations in \S\ref{subsec:disc bundle hopf:p.s.c nbd basis}.
As we mentioned, the Hopf surface $H$ does not admit any K\"{a}hler metric.
So it is very natural to ask whether the flat disc bundle $\Omega$ constructed in \S\ref{subsec:disc bundle hopf:p.s.c nbd basis}
admits a K\"{a}hler metric. In this subsection, we will prove that $\Omega$ admits a complete K\"{a}hler metric.
The essential observation here is that $\Omega$ is biholomorphic to a product.

We have seen that $\Omega$ is foliated by closed submanifolds $X_v$ that are all biholomorphic to $\mc^2_*$ with $v\in \Delta$,
and $X_v$ and $X_w$ are equal if and only if $v, w$ lie in the same orbit with respect to the action of $\Gamma'$ on $\Delta$.
However, just from this fact, it is not obvious whether $\Omega$ is biholomorphic to a $\mc^2_*$ bundle over $\Delta/\Gamma'$.
But we will show that it is indeed the case.

Note that $\Gamma'$ acts on $\Delta$ freely and properly discontinuously.
The quotient space $A=\Delta/\Gamma'$ is a Riemann surface which is clearly not compact.

\begin{lem}\label{lem:quotient annulus}
The quotient space $A$ is holomorphically equivalent to an annulus $A_r:=\{v\in\mc; r<|v|<1\}$ for some $r\in(0,1)$.
\end{lem}
\begin{proof}
Note that the fundamental group of $A$ is isomorphic to $\Gamma'$, which is commutative.
By Theorem IV.6.8 in \cite{Farkas-Kra}, $A$ is biholomorphic to some $A_r$ or the punctured disc $\Delta^*$.
Note that the automorphism group of $\Delta^*$ is just the circle group,
and the automorphism group of $A_r$ have two connected components.
Let $G=\{\phi_r(z)=\frac{z+r}{1+rz}; r\in(-1,1)\}\subset Aut(\Delta)$.
Then $G/\Gamma'$ is naturally identified to a subgroup of $Aut(A)$ which is clearly isomorphic to $S^1$.
Let $\tau\in Aut(\Delta)$ be given by $\tau(z)=-z$.
Then we have $\tau\phi=-\phi^{-1}$. So $\tau$ also induces an automorphism of $A$,
which is also denoted by $\tau$.
It is easy to check that $\phi_r\circ\tau^{-1}=\phi_r\circ\tau=-\phi_{-r}\notin G$
for any $r\neq 0$. So $\tau\notin G/\Gamma'$, which implies that $Aut(A)$ is not connected.
So $A$ has to be biholomorphic to some $A_r$.
\end{proof}
% it's interesting to determine $r$ explicitly.

Recall the construction of $\Omega$.
We have a group $\Gamma$ acting freely and properly discontinuously on $\mc^2_*$,
and a group $\Gamma'$ acting freely and properly discontinuously on $\Delta$.
The map $\sigma:\Gamma\rightarrow\Gamma'$ given by $\gamma_n\mapsto \phi^{n}$ is a group isomorphism
and induces an action of $\Gamma$ on $\Delta$, and further induces an action of $\Gamma$ on $\tilde\Omega=\mc^2_*\times \Delta$.
Then $\Omega:=\tilde\Omega/\Gamma$ is naturally realized as a flat $\Delta$-bundle over the Hopf surface $H=\mc^2_*/\Gamma$.

Now the key observation is that $\Omega$ can also be constructed by revising the roles of $\mc^2_*$ and $\Delta$,
and the roles of $\Gamma$ and $\Gamma'$.
Note that the group isomorphism $\sigma^{-1}:\Gamma'\rightarrow\Gamma$ induces an action of $\Gamma'$ on $\mc^2_*$
and further induces an action of $\Gamma'$ on $\Delta\times\mc^2_*$, which is of course can be identified to $\mc^2_*\times\Delta$.
Now it is obvious that $\Omega$ can be naturally identified to $\Delta\times\mc^2_*/\Gamma'$.
But now by the construction in \S\label{subsec:framework}, $\Omega$ is realized as a flat $\mc^2_*$-bundle over $A=\Delta/\Gamma'$.
Since the action of $\Gamma$ on $\mc^2_*$ is linear,
$\Omega$ can be extended to a rank two holomorphic vector bundle, say $E$, over $A$ by adding the zero section $E_0$ to $\Omega$.

We now apply the famous Oka-Grauert principle to $E$,
which states that $E$ is trivial as a holomorphic vector bundle if and only if it is trivial as a complex topological bundle (see \cite{Forstneric11}).
Note that the isomorphism class of a complex topological bundle of rank 2 over $A$ is determined by a homotopic class of continuous maps from $A$ to
the classifying space $BU(2)=G_2(\mc^\infty)$ (see Theorem 23.10 in \cite{Bott-Tu82}).
Note that $BU(2)$ is simply connected.
Since $A$ is biholomorphic to some $A_r$ by Lemma \ref{lem:quotient annulus}, all continuous maps from $A$ to $BU(2)$
are homotopically trivial.
So all complex topological vector bundle over $A$ are trivial.
Hence $E$ is a holomorphically trivial vector bundle over $A$.
We get the following

\begin{thm}\label{thm:product structure}
$\Omega$ is biholomorphic to the product manifold $A\times\mc^2_*$.
\end{thm}

Note that $A\times\mc^2$ is a Stein manifold, by Lemma \ref{lem:Kahler metric complement},
we get

\begin{thm}
There is a complete K\"{a}hler metric on $\Omega$.
\end{thm}

The $\mathbb P^1$-bundle $B$ over $H$ has no K\"{a}hler metric since it contains
the compact submanifolds $X_{\pm 1}$, which are biholomorphic to the Hopf surface.
Furthermore it seems that $B$ is not bimeromorphic to a compact K\"{a}hler manifold.
It is also unknown whether there exists a fiber bundle over $H$ with compact fibers
such that the total space is K\"{a}hler.

\subsection{Disc bundle over the Hopf surface III}\label{subsec:disc bundle hopf:hyperconvex}
In this subsection, we construct a disc bundle over the Hopf surface $H$
such that the total space has a real analytic exhaustive bounded plurisubharmonic function
but has no nonconstant holomorphic functions.
We will follow the general idea discussed in \S \ref{subsec:framework}.

We again denote $\mc^2\backslash 0$ by $\mc^2_*$,
and let $\Gamma=\{\gamma_n:z\mapsto 2^nz;n\in \mathbb Z\}\subset Aut(\mc^2_*)$.
The quotient space $H=\mc^2_*/\Gamma$ is the Hopf surface.

Fix an irrational real number $\lambda$ and let $\alpha=e^{\lambda 2\pi i}$.
Then $\phi(v)=\alpha v$ is a rotation of $\mathbb P^1$ and let $\Gamma'=\{\phi^n;n\in\mathbb Z\}$.
Note that $\Gamma'$ can also be viewed as an automorphism group of any disc centered at the origin.

The map $\sigma:\Gamma\rightarrow\Gamma'; \ \gamma_n\mapsto \phi^n$ is a group isomorphism,
and induces an action of $\Gamma$ on $\mathbb P^1$, and further induces an action of $\Gamma$ on the product $\tilde B=\mc^2_*\times \mathbb P^1$.
Let $B=\tilde B/\Gamma$, which is a flat $\mathbb P^1$-bundle over $H$.
Let $\tilde\Omega=\mc^2_*\times\Delta$ and let $\Omega=\tilde\Omega/\Gamma$,
then $\Omega$ is a disc bundle over $H$ and is a domain in $B$ with real analytic pseudoconvex boundary.
Our aim is to prove the following

\begin{prop}
There is a real analytic bounded exhaustive plurisubharmonic function on $\Omega$;
but there is no nonconstant holomorphic function on $\Omega$.
\end{prop}
\begin{proof}
Let $\tilde\rho:\tilde\Omega\rightarrow\mr$ be the function given by $(z,v)\mapsto |v|^2$.
It is clear that $\tilde\rho$ is a $\Gamma$-invariant p.s.h function on $\tilde\Omega$.
So it induces a p.s.h function on $\Omega$ which is exhaustive on $\Omega$.

For the second statement, assume $f$ is a holomorphic function on $\Omega$.
Let $\tilde f$ be the pull back of $f$ on $\tilde\Omega=\mc^2_*\times \Delta$.
Then $\tilde f$ is invariant under the action of $\Gamma$.
More precisely, for all $(z,v)\in\mc^2_*\times\Delta$, for any $n\in \mathbb Z$,
we have $\tilde f(2^nz,\alpha^nv)=\tilde f(z,v)$.

By Riemann's removable singularity theorem, $\tilde f$ can be extended to a holomorphic function
on $\mc^2\times\Delta$. We expand $\tilde f$ near (0,0,0) as
$$\tilde f(z_1, z_2, v)=\sum_{i,j,k\geq 0}a_{ijk}z^i_1z^j_2v^k.$$
Then we have
$$\tilde f(2^nz_1, 2^nz_2, \alpha^n v)=\sum_{i,j,k\geq 0}a_{ijk}2^{n(i+j)}\alpha^{nk}z^i_1z^j_2vk.$$
From the invariance of $\tilde f$, we get
$$a_{ijk}=a_{ijk}2^{n(i+j)}\alpha^{nk},\; \text{for\ all}\ i,j,k\geq 0,\ n\in\mathbb Z.$$
Since $\alpha$ is irrational, we get $a_{ijk}=0$ if $i+j+k>0$,
and hence $\tilde f$ is constant.
So $f$ is constant.
\end{proof}

\subsection{$\mc^*$-bundle over the Hopf surface}\label{subsec:C^* bundle and Kahler metric}
In this subsection, we construct
a flat $\mc^*$-bundle over the Hopf surface $H$ such that the total space has a complete K\"{a}hler metric.

Recall that $\Gamma=\{\gamma_n:z\mapsto 2^nz; n\in\mathbb Z\}$ acts freely and properly discontinuously on $\mc^2_*$.
We define an action of $\Gamma$ on $\mathbb P^1$ by associating $\gamma_n$ to the map $v\mapsto 2^nv, v\in\mathbb P^1$.
Then $\mc^*\subset\mathbb P^1$ is a $\Gamma$-invariant open subset and
$\Gamma$ acts on $\mc^*$ freely and properly discontinuously.
The quotient space $S=\mathbb C^*/\Gamma$ is an elliptic curve.

We can now define an action of $\Gamma$ on $\tilde B=\mc^2_*\times\mathbb P^1$
by acting on both factors. Let $B=\tilde B/\Gamma$ be the quotient space.
Let $\tilde\Omega=\mc^2_*\times\mc^*$ and $\Omega=\tilde\Omega/\Gamma$.
Similarly to the consideration in \S\ref{subsec:disc bundle hopf:Kahler metric},
we can look at $\Omega$ in two ways.
One is to view it as a $\mc^*$-bundle over the Hopf surface $H$,
and the other is to view it as a $\mc^2_*$ bundle over $S$.
If taking the second viewpoint, $\Omega$ can be extended to a rank two holomorphic
vector bundle say $E$ over $S$ by adding the zero section.

By Lemma \ref{lem:general existece Kahler metric}, $E_\infty$ (see \S \ref{subsec:framework} for definition) is a compact K\"ahler manifold.
Note that $\Omega\subset E_\infty$ is the complement of an analytic set in $E_\infty$.
By Lemma \ref{lem:Kahler metric complement}, $\Omega$ has a complete K\"{a}hler metric.

\begin{thm}
The flat $\mc^*$-bundle $\Omega$ over the Hopf surface $H$ has a complete K\"{a}hler metric.
\end{thm}

We take a more careful look at $B$, which is a compactification of $\Omega$.
$\Omega$ is a Zariski open set of $B$ and the complement $B\backslash\Omega=X_0\cup X_\infty$,
where $X_0=\{[z,0];z\in\mc^2_*\}$ and $X_\infty=\{[z,\infty];z\in\mc^2_*\}$
are two copies of the Hopf surface $H=\mc^2_*/\Gamma$.
On the other hand, by identifying $\mathbb P^2=\mc^2\cup\mathbb P^1$,
we can compactify $E$ naturally to a $\mathbb P^2$-bundle $E_\infty$ over $S$ as shown in \S \ref{subsec:framework}.
$E_\infty$ is a compactification of $\Omega$ that is different from $B$.
The complement $E_\infty\backslash \Omega$ of $\Omega$ in $E_\infty$ also
contains two connected components, namely the zero section $E_0$ of $E$
and $D_\infty:=E_\infty\backslash E$ which is a $\mathbb P^1$-bundle over the elliptic curve $S$.

It is interesting to compare $X_\infty$ and $D_\infty$,
which are connected boundary components in the two different compactifications of $\Omega$.
As mentioned above, $D_\infty$ is a $\mathbb P^1$-bundle over $S$.
Recall that $H=\mc^2_*/\Gamma$, so it can be fibred over $\mc^2_*/\mc^*$
with fibers biholomorphic to $\mc^*/\Gamma=S$.
So $X_\infty$ is an $S$-bundle over $\mathbb P^1$.
So it seems that $X_\infty$ and $D_\infty$ has some duality relation,
which deserves a further  study.

There is no K\"{a}hler metric on $B$ since it contains $H$ as a submanifold.
But $E_\infty$ is projective manifold by Lemma \ref{lem:general existece Kahler metric}.
One can prove that $B$ and $E_\infty$ are not bimeromorphic.
In fact, if $B$ can not be bimeromorphic to any K\"ahler manifold.
If this is not true, there will be a K\"ahler current say $T$ on $B$.
It is then clear that the direct image $p_*T$ is a K\"ahler current on $H$, 
recalling that $B$ is a fiber bundle over $H$ and $p:B\rightarrow H$ is the projection.
Then $p_*T$ gives a nonzero class in $H^2(H)$.
This is a contradiction since the second cohomology group of $H$ vanishes. 

\section{Flat bundles over compact Riemann surfaces}

\subsection{Flat bundles over elliptic curves}\label{subsec:flat bundle elliptic curve}
The aim is to construct compact projective manifolds $B$ and some Stein domain $\Omega$ in $D$
with real analytic boundary such that the closure of $\Omega$ in $B$ has no pseudoconvex neighborhood basis.
We will realize $B$ as a flat $\mathbb P^1$ bundle over elliptic curves.

Let $\Gamma$ be a lattice in $\mc$, then $\Sigma=\mc/\Gamma$ is an elliptic curve.
Let $\tilde\sigma:\Gamma\rightarrow\mathbb Z$ be a surjective group morphism.
Then $\sigma:\Gamma\rightarrow Aut(\mathbb P^1)$ given by $\gamma\mapsto \phi^{\tilde\sigma(\gamma)}$
is a group morphism, where $\phi(v)=\frac{v+1/2}{1+v/2}$ is a automorphism of $\mathbb P^1$.
Through $\sigma$, $\Gamma$ acts on $\mathbb P^1$ and further acts on $\tilde B:=\mc\times\mathbb P^1$ by acting on each factor.
The quotient space $B=\tilde B/\Gamma$ is a complex manifold, which is a $\mathbb P^1$-bundle over $\Sigma$.
Let $\tilde\pi:\tilde B\rightarrow B$ and $\pi:\mc\rightarrow \Sigma$ be the quotient maps,
and let $\tilde p:\tilde B\rightarrow\mc$ and $p:B\rightarrow\Sigma$ be the bundle maps.
We have $p\circ\tilde\pi=\pi\circ\tilde p$.

Since the action of $\Gamma$ on $\mathbb P^1$ fix the $\pm 1$,
by the same argument as in proof of Lemma \ref{lem:general existece Kahler metric},
we can prove that $B$ is a projective manifold.

Let $\Delta^*=\mathbb P^1\backslash\overline\Delta$.
Let $\tilde\Omega=\mc\times\Delta, \tilde\Omega^*=\mc\times\Delta^*$.
Let $\Omega=\tilde\Omega/\Gamma, \Omega^*=\tilde\Omega^*/\Gamma$, which are disc bundles over $\Sigma$.
The boundary of $\Omega$ and $\Omega^*$ in $B$ are equal and equal to $\mc\times S^1/\Gamma$,
which is a real analytic Levi-flat hypersurface in $B$.

The quotient space $\mc/\ker\sigma$ is biholomorphic to $\mc^*$,
and the action of $\Gamma/\ker\sigma$ on $\mc/\ker\sigma$ corresponds to multiplications on $\mc^*$.
By the discussion in \S \ref{subsec:framework},
 $\Omega$ and $\Omega^*$ are biholomorphic to $\mc^*$ bundles over the annulus $A:=\Delta/\sigma(\Gamma)$,
whose transition functions are given by multiplications on $\mc^*$.
So $\Omega$ and $\Omega^*$ can be extended to holomorphic line bundles over $A$ by adding zero sections.
By the Oka-Grauert principle and by the same argument as in the proof of Theorem \ref{thm:product structure},
both $\Omega$ and $\Omega^*$ are biholomorphic to the product $A\times\mc^*$,
and hence are Stein.

\begin{thm}\label{thm:no psc nbd basis elliptic curve}
The flat $\mathbb P^1$-bundle $B$ over $\Sigma$ is projective.
The domain $\Omega$ in $B$ is Stein and its boundary is a real analytic Levi-flat hypersurface in $B$.
$\overline\Omega$ has no pseudoconvex neighborhood basis in $B$.
\end{thm}
\begin{proof}
The first statement follows from the above discussion.
For the last statements, assume $D$ is a small pseudoconvex neighborhood of $\overline\Omega$ in $B$.
Then $D\cap\Omega^*$ is a pseudoconvex domain in $\Omega^*$.
By construction, $\Omega^*\backslash D$ is a compact set.
This is impossible since $\Omega^*$ is Stein.
\end{proof}

\subsection{Flat bundles over compact Riemann surfaces of positive genus}
Applying Mok's solution of the Serre problem for open Riemann surfaces \cite{Mok81},
we can generalize the construction in \S\ref{subsec:flat bundle elliptic curve} to all compact Riemann surfaces of positive genus.
The aim is to construct compact projective manifolds $B$ and some Stein domain $\Omega$ in $B$
with real analytic boundary such that the closure of $\Omega$ in $B$ has no pseudoconvex neighborhood basis.
Here $B$ is realized as a flat $\mathbb P^1$-bundle over a compact Riemann surfaces and $\Omega$ is the
corresponding disc bundle in $B$.

Let $X$ be a compact Riemann surface with genus $g\geq 1$.
Let $\tilde X$ be the universal covering of $X$ with $X=\tilde X/\Gamma'$,
where $\Gamma'\subset Aut(\tilde X)$ is isomorphic to the fundamental group of $X$.
By the uniformization theorem, $\tilde X$ is biholomorphic to $\mc$ if $g=1$ and biholomorphic to $\Delta$ if $g>1$.

It is well known from topology that $H_1(X,\mathbb Z)\approx \Gamma'/[\Gamma',\Gamma']$,
where $[\Gamma',\Gamma']$ is the normal subgroup of $\Gamma'$ generated by all elements of the form
$aba^{-1}b^{-1}$ with $a,b\in\Gamma'$ (see Theorem 2A.1 in \cite{Hatcher01}).
Note that $H^1(X,\mathbb Z)\approx \mathbb Z^{2g}$, there is a surjective group morphism $\sigma':\Gamma'\rightarrow\mathbb Z$.

Let $D=\tilde X/\ker\sigma'$ and let $\Gamma=\Gamma'/\ker\sigma'$.
Then $D$ is an open Riemann surface and  $\Gamma$ is isomorphic to $\mathbb Z$ which acts freely and properly discontinuously on $D$.
The quotient space $D/\Gamma$ is just $X$.

Let $\gamma$ be a generator of $\Gamma$.
Let $\sigma:\Gamma\rightarrow Aut(\mathbb P^1)$ be the group morphism given by $\gamma^n\mapsto \phi^n$,
where $\phi(v)=\frac{v+1/2}{1+v/2}$ is an automorphism of $\mathbb P^1$.
Then $\sigma$ induces an action of $\Gamma$ on $\mathbb P^1$ and further induces an action of $\Gamma$ on $\tilde B=D\times\mathbb P^1$.
The quotient $B=\tilde B/\Gamma$ if a flat $\mathbb P^1$-bundle over $X$.

Let $\Delta^*=\mathbb P^1\backslash\overline\Delta$.
Let $\tilde\Omega=D\times\Delta, \tilde\Omega^*=D\times\Delta^*$.
Let $\Omega=\tilde\Omega/\Gamma, \Omega^*=\tilde\Omega^*/\Gamma$, which are disc bundles over $X$.
The boundary of $\Omega$ and $\Omega^*$ in $B$ are equal and equal to $D\times S^1/\Gamma$,
which is a real analytic Levi-flat hypersurface in $B$.

\begin{lem}\label{lem:projectivity no vector bd}
$B$ is a projective manifold.
\end{lem}
\begin{proof}
Viewing $B$ as a $\mathbb P^1$-bundle over $X$, the transition functions preserves $1\in\mathbb P^1$.
Let $L_1$ be the line bundle over $B$ corresponding to the divisor $X\times\{ 1\}$.
Then $L_1$ is positive when restricted to each fiber of $B\rightarrow X$.
By Proposition \ref{prop:submersion case}, $B$ is projective.
\end{proof}

\begin{lem}
Both $\Omega$ and $\Omega^*$ are Stein manifolds.
\end{lem}
\begin{proof}
It is clear that $\Omega$ and $\Omega^*$ are biholomorphic.
From the discussion in \S \ref{subsec:framework}, $\Omega$ is a flat disc bundle over $X$.
Since $\sigma(\Gamma)=\{\phi^n;n\in\mathbb Z\}$ acts freely and properly discontinuously on $\Delta$,
$\Omega$ can also be realized as a flat $D$-bundle over $A=\Delta/\Gamma$.
We have proven that $A$ is biholomorphic to an annulus in $\mc$ in Lemma \ref{lem:quotient annulus}.
By Mok's solution of the Serre's problem for open Riemann surfaces in \cite{Mok81},
$\Omega$ is a Stein manifold.
\end{proof}

\begin{thm}
The closure $\overline\Omega$ of $\Omega$ in $B$ has no pseudoconvex neighborhood basis.
\end{thm}
\begin{proof}
Assume that $W$ is a small pseudoconvex neighborhood of $\overline\Omega$ in $B$.
Then $W\cap\Omega^*$ is a pseudoconvex domain in $\Omega^*$.
By construction, $\Omega^*\backslash W$ is a compact set.
This is impossible since  $\Omega^*$ is Stein.
\end{proof}

\section{Flat bundles over compact manifolds with positive first Betti numbers}
In this section, we will prove that for any compact complex manifold $X$ with positive first Betti number,
there is a flat disc bundle over $X$ such that its total space has a real analytic exhaustive p.s.h function
but does not admit any nonconstant holomorphic function.
The construction is a generalization of that in \S\ref{subsec:disc bundle hopf:hyperconvex}.

Let $X$ be a compact complex manifold with first Betti number $b_1(X)\geq 1$.
Let $\tilde X$ be the universal covering of $X$ with $X=\tilde X/\Gamma'$,
where $\Gamma'\subset Aut(\tilde X)$ is isomorphic to the fundamental group of $X$.

We have $H_1(X,\mathbb Z)\approx \Gamma'/[\Gamma',\Gamma']$,
where $[\Gamma',\Gamma']$ is the normal subgroup of $\Gamma'$ generated by all elements of the form
$aba^{-1}b^{-1}$ with $a,b\in\Gamma'$ (see Theorem 2A.1 in \cite{Hatcher01}).
Since $b_1(X)\geq 1$, there is a surjective group morphism $\sigma':\Gamma'\rightarrow\mathbb Z$.

Let $D=\tilde X/\ker\sigma'$ and let $\Gamma=\Gamma'/\ker\sigma'$.
Then $D$ is a noncompact manifold and  $\Gamma$ is isomorphic to $\mathbb Z$ which acts freely and properly discontinuously on $D$.
The quotient space $D/\Gamma$ is just $X$.

Let $\gamma$ be a generator of $\Gamma$.
Let $\sigma:\Gamma\rightarrow Aut(\mathbb P^1)$ be the group morphism given by $\sigma(\gamma)(v)=\alpha v$,
where $\alpha=e^{\lambda2\pi i}$ for a fixed irrational real number $\lambda$.
Then $\sigma$ induces an action of $\Gamma$ on $\mathbb P^1$ and further induces an action of $\Gamma$ on $\tilde B=D\times\mathbb P^1$.
The quotient $B=\tilde B/\Gamma$ if a flat $\mathbb P^1$-bundle over $X$.

Let $\tilde\Omega=D\times\Delta, \tilde\Omega^*=D\times\Delta^*$.
Let $\Omega=\tilde\Omega/\Gamma$, which is a disc bundle over $X$.
The aim is to prove the following

\begin{thm}
$B$ is a projective manifold if $X$ is projective.
There is a real analytic exhaustive plurisubharmonic function on $\Omega$.
There is no nonconstant holomorphic function on $\Omega$.
\end{thm}
\begin{proof}
The first statement can be proved by the same argument as in the proof of Lemma \ref{lem:projectivity no vector bd}.

For the second statement, let $\tilde\rho(z,v)=|v|^2$.
Then $\tilde\rho$ is a real analytic function on $\tilde\Omega=D\times\Delta$ which is invariant under the action of $\Gamma$,
and hence induces a plurisubharmonic function say $\rho$ on $\Omega$,
which is obviously exhaustive.

For the last statement,
since $D/\Gamma$ is compact, there is a compact set $F\subset D$ such that $\cup_{n\in\mathbb Z}\gamma^nF=D$.
Let $D_v=D\times\{v\}\subset \tilde\Omega$ for $v\in\Delta$ and  $S_r=\{v'\in\Delta; |v'|=r\}$ for $r\in (0,1)$.
Assume $f$ is a holomorphic function on $\Omega=\tilde\Omega/\Gamma$.
Let $\tilde f$ be the pull back of $f$ to $\tilde \Omega$.
Then $\tilde f$ is invariant under the action of $\Gamma$,
that is, $\tilde f(\gamma^n z, \alpha^nv)=\tilde f(z,v), (z,v)\in D\times \Delta$ for all $n\in \mathbb Z$.
It suffices to prove that $\tilde f$ is constant.

Let $(z_0,v_0)\in F\times S_r$ such that $\tilde f(z_0,v_0)=\sup_{(z,v)\in F\times S_r}|\tilde f(z,v)|$.
For any $(z,v_0)\in D_{v_0}$, there exists $n$ such that $(\gamma^nz,\alpha^nv_0)\subset F\times S_r$.
Note that $\tilde f$ is invariant under the action of $\Gamma$,
$\tilde f|_{D_{v_0}}$ attains its maximum at $(z_0,v_0)$.
By the maximum principle for holomorphic functions, $\tilde f|_{D_{v_0}}$ is constant and identically equals to $\tilde f(z_0,v_0)$.

We go further to prove that $\tilde f$ is constant on $D\times S_r$.
By assumption, we also have $\tilde f(z_0,v_0)=\sup_{(z,v)\in D\times S_r}|\tilde f(z,v)|$.
For any $n\in\mathbb Z$, we have $\tilde f(\gamma^n z_0, \alpha^n v_0)=\tilde f(z_0,v_0)$.
So $\tilde f|_{D_{\alpha^n v_0}}$ attains its maximum $\tilde f(z_0,v_0)$ at $(\gamma^n z_0, \alpha^n v_0)$.
Again by the maximum principle, $\tilde f|_{D_{\alpha^n v_0}}\equiv \tilde f(z_0, v_0)$.
So $\tilde f$ takes constant value on $\cup_{n\in\mathbb Z}D_{\alpha^n v_0}$.
Since $\lambda$ is irrational, $\{\alpha^nv_0;n\in\mathbb Z\}$ is dense in $S_r$.
By continuity, $\tilde f$ is constant on $D\times S_r$.
By the identity theorem for holomorphic functions, $\tilde f$ is constant on $D\times\Delta$.
\end{proof}

In the above construction, we can replace $\Delta$ by the unit ball $B^n$ of dimension $n\geq 1$,
and replace the action of $S^1$ on $\Delta$ by the natural action of the $n$-dimensional torus $T^n=S^1\times\cdots\times S^1$ on $B^n$.
Note that there exist $\alpha_1, \cdots, \alpha_n\in S^1$ such that $\{(\alpha^k_1,\cdots,\alpha^k_n)|k\in\mathbb Z\}$ is dense in $T^n$.
By the same argument, we can prove the following theorem which answers affirmatively a problem proposed by professor Xiangyu Zhou (private communication).

\begin{thm}
For any compact complex manifold $X$ with positive first Betti number and any $n\geq 1$,
there is a flat $B^n$-bundle $\Omega$ over $X$ such that:
\begin{itemize}
\item[(1)] there is a real analytic bounded exhaustive plurisubharmonic function $\rho$ on $\Omega$ such that $\partial\bar\partial \rho$
has $n$-positive eigenvalues;
\item[(2)] $\Omega$ admits no nonconstant holomorphic function.
\end{itemize}
\end{thm}

\bibliographystyle{amsplain}

\begin{thebibliography}{123}
\bibitem{Adachi17}M. Adachi, Weighted Bergman spaces of domains with Levi-flat boundary: geodesic segments on compact Riemann surfaces , preprint, arXiv:1703.08165.
\bibitem{Bott-Tu82} R. Bott, L.W. Tu, Differential forms in algebraic topology, Springer-Verlag, New York, 1982.
\bibitem{Diederich-Fornaess77}K. Diederich and J. E. Forn{\ae}ss, Pseudoconvex Domains: An Example with Nontrivial
          Nebenhulle. Math. Ann. 225(1977), 275-297.
\bibitem{Diederich-Fornaess78}K. Diederich and J. E. Forn{\ae}ss, Pseudoconvex Domains with Real-Analytic Boundary, Ann. Math., 107(1978), 371-384.
\bibitem{Diederich-Fornaess82}K. Diederich and J. E. Forn{\ae}ss, A smooth pseudoconvex domain without pseudoconvex exhaustion. Manuscripta Math. 39 (1982), 119-123.
\bibitem{Demailly} J. P. Demailly, Estimations $L^2$ pour l'op\'erateur $\bar\partial$ d'un fibr\'e vectoriel holomorphe semipositif
          au dessus d'une vari\'et\'e K\"ahl\'erienne compl\`{e}te, Ann. Sci. Ecole Norm. Sup. 15 (1982) 457-511.
\bibitem{Farkas-Kra} H. M. Farkas and I. Kra, Riemann Surfaces, Second Edition, Springer-Verlag New York, 1992.
\bibitem{Forstneric11}F. Forstneri\v{c}, Stein Manifolds and Holomorphic Mappings, Ergebnisse der
                  Mathematik und ihrer Grenzgebiete, 3. folge 56. Springer-Verlag, Berlin-Heidelberg 2011.
\bibitem{Hatcher01} A. Hatcher, Algebraic topology, Cambridge University Press, 2001.
\bibitem{Mok81} N. Mok, The Serre problem on Riemann surfaces. Math. Ann. 258 (1981/82), 145-168.








\end{thebibliography}

\end{document}